\documentclass[a4paper,11pt]{article}
\usepackage{amsmath}
\usepackage{amssymb}
\usepackage{mathrsfs}
\usepackage{amsfonts}
\usepackage{amsthm}
\usepackage{pstricks, pst-node, pst-text, pst-3d,psfrag}
\usepackage{graphicx}
\usepackage{indentfirst}
\usepackage{enumerate}
\usepackage{subfigure}
\usepackage{cite}
\usepackage[colorlinks=true]{hyperref}
\hypersetup{urlcolor=blue, citecolor=red}

\theoremstyle{plain}
\newtheorem{thm}{Theorem}[section]

\newtheorem{lem}[thm]{Lemma}
\theoremstyle{definition}

\theoremstyle{remark}
\newtheorem{rmk}[thm]{\textbf{Remark}}

\numberwithin{equation}{section}

\makeatletter % `@' now normal "letter"
\@addtoreset{equation}{section}
\makeatother  % `@' is restored as "non-letter"

\newcommand\abs[1]{\lvert#1\rvert}

\newcommand\RR{\ensuremath{\mathbb{R}}}

\def \g{\gamma}
\begin{document}

\title{Non-Oscillation Principle for Eventually Competitive and Cooperative Systems}

\setlength{\baselineskip}{16pt}

\author {
Lin Niu and Yi Wang\thanks{Partially supported by NSF of China No.11771414, 11471305 and Wu Wen-Tsun Key Laboratory.} \\
School of Mathematical Science\\
 University of Science and Technology of China
\\ Hefei, Anhui, 230026, P. R. China
}
\date{}

    \maketitle
% insert the table of contents
%\tableofcontents

%---------------------SECTION DIVIDE LINE---------------------------
\begin{abstract}
A nonlinear dynamical system is called eventually competitive (or cooperative) provided that it preserves a partial order in backward (or forward) time only after some reasonable initial transient.
We presented in this paper the Non-oscillation Principle for eventually competitive or cooperative systems, by which the non-ordering of (both $\omega$- and $\alpha$-) limit sets is obtained for such  systems; and moreover, we established the Poincar\'{e}-Bendixson Theorem and structural stability for three-dimensional eventually competitive and cooperative systems.
\end{abstract}

%\textit{\textbf{Keywords:}} eventually cooperative, eventually competitive, non-oscillation principle, non-ordering, Poincar\'{e}-Bendixson theorem, structural stability

\section{Introduction}

A system of differential equations in $N$-space ($N\geqslant 1$) is called competitive (or cooperative) provided that all the off-diagonal entries of its linearized Jacobian matrix are nonpositive (or nonnegative).
It is now well-known that the flow of a cooperative (or competitive) system preserves the vector partial order in forward (or backward) time, by which Hirsch initiated an important research branch of so called monotone (or competitive) dynamical systems. One may refer to the monographs and recent reviews \cite{HS03,HS05,S95,S17} with references therein for the theoretical developments and their enormous applications to control, biological and  economic systems (cf. \cite{AS03,L86,Sontag07}).

It was later found that there are larger classes of systems whose flows may preserve the vector order in forward (or backward) time only after some reasonable initial transient.  Following Hirsch \cite{H85}, the flow $\phi_{t}$ generated by such a system is called as an {\it eventually cooperative} (resp. {\it competitive}) flow, i.e., there exists some $t_*\geqslant 0$ such that $\phi_{t}(x)\leqslant \phi_{t}(y)$ whenever $x\leqslant y$ and $t\geqslant t_*$ (resp. $t\leqslant -t_*$)
\footnote{
In particular,  $\phi_{t}$ is cooperative (resp. competitive) if $t_*=0$. In contrast, if $t_*>0$, then it does not require any order-preserving information for $t\in [0,t_*)$ (resp. $t\in (-t_*,0]$) at all. Consequently, the eventually cooperative (or competitive) systems defined here is totally different from the more established definition of eventual strong monotonicity in the literatures (see, e.g. \cite[Section 4.3]{HS05} or \cite[Sec. 5.3]{S95}).
}.

In the terminology of linear systems, for instance, such phenomenon is often also called {\it eventual positivity} (see \cite{DGK16JDE,NT08,SM17} and references therein) in forward (or backward) time, which means that trajectories starting from  positive initial values will become positive in forward (or backward) time only after some initial transient.
As a matter of fact, this property received rapidly-increasing attention very recently in both finite-dimensional linear systems \cite{NT08} and infinite-dimensional linear systems \cite{DGK16,DGK16JDE}, as well as applications to ordinary differential equations \cite{OTD09,SWo91}, partial differential equations \cite{D14,FGG08,GG08}, delay differential equations \cite{DGK16,DGK16JDE} and control theory \cite{A15,AL15}.

For nonlinear systems, Hirsch \cite{H85} has observed that the regular perturbation of a cooperative irreducible vector field is at most eventually cooperative rather than cooperative. Sontag and Wang \cite{WS08} later showed that singular perturbations of cooperative irreducible vector fields are just eventually cooperative as well.
As a consequence, in order to understand the impact of perturbations (either regular or singular) on cooperative (or competitive) systems, one needs to investigate the dynamical properties of the eventually cooperative (or competitive) systems.

Very recently, various examples of systems have been found in \cite{SM17}, which cannot be confirmed to be monotone (cooperative), but are eventually monotone (eventually cooperative). Moreover,
they are
not limited to near monotone (cooperative) systems in the context of perturbation theory. Since
these examples describe different types of
biological and biomedical processes, they exhibit the significance for developing the theory of eventual cooperative (or competitive) systems.

Hirsch \cite{H85} has ever studied the dynamics of eventually monotone (cooperative) systems. Sootla and Mauroy \cite{SM17} proposed a spectral characterization of eventually monotone systems from Koopman operators point of view. Based on all these results, it is reasonable to expect that eventually cooperative (or competitive) systems might possess many asymptotic properties of cooperative (or competitive) systems.

The fundamental building block of the theory
of cooperative (or competitive) systems is the Non-ordering of Limit Sets (see e.g., Smith \cite[Theorem 2.1]{S17}). In both discrete-time and continuous-time systems, the remarkable generic convergence for irreducible cooperative (monotone) systems strongly relies on the Non-ordering of Limit Sets. Therefore,
a natural question is whether Non-ordering of Limit Sets can still hold for systems which are eventually cooperative or competitive.

Hirsch \cite{H85} first tackled this question and proved the non-ordering of the {\it $\omega$-limit sets} for eventually cooperative systems. However, it deserves to point out that Hirsch's proof for the non-ordering of the $\omega$-limit sets was based on the so-called monotone convergence criterion that he discovered for eventually cooperative systems (see \cite[Theorem 2.2]{H85}). When one encounters the $\omega$-limit sets for eventually competitive systems; or equivalently, the $\alpha$-limit sets for eventually cooperative systems, the monotone convergence criterion does not work anymore. Consequently, one may recall a more general effective tool, called {\it Non-oscillation Principle} (see \cite[Lemma 6.1]{H88}), which has played a key role for guaranteeing the non-ordering of (both $\omega$- and $\alpha$-) limit sets for cooperative and competitive systems. Nevertheless, as we will explain below, it is by no means obvious to prove Non-oscillation Principle for the eventually cooperative or competitive systems.

Roughly speaking, Non-oscillation Principle means that any trajectory cannot oscillate with respect to a partial ordering ``$\leqslant$" between the vectors in $\RR^N$. More precisely, let $x(t)$ be a trajectory of $x$ for cooperative or competitive systems. A subinterval $J=[a,b]$ is called an increasing-interval if the ending points of $J$ satisfy $x(a)<x(b)$, and a decreasing-interval if $x(a)>x(b)$ (see the definition for ``$\leqslant,<,>$" in Section 2).

\vskip 2mm
\noindent$ \bullet$ (Non-oscillation Principle). {\it The trajectory of $x$ cannot have both an increasing-interval and a decreasing-interval which are disjoint.}
\vskip 2mm

In the existing literatures there are two main approaches to prove the Non-oscillation Principle for  cooperative and competitive systems. Let us call these two approaches as ``{\it Finite-Interval Approach}" and ``{\it Continuation Argument Approach}", respectively. Both of them are based on proof by contradiction.
In the following,
we will point out the critical points of each of the approaches, and further explain briefly the difficulties encountered in these approaches when one deals with the eventually cooperative or competitive systems.

Before doing that, we call that a subinterval $J=[a,b]$ is {\it steeply increasing} for the trajectory of $x$ if $J$ is increasing and the ending point $a$ is the only point $t\in J$ with $x(t)\leqslant x(a)$.

The first approach (i.e., Finite-Interval Approach) is originated in Hirsch \cite[Proposition 2.5]{H82}, who attributes the proof to L. Ito. Since then, this approach turned out to be the popular way to establish the Non-oscillation Principle in the literatures; and it was later improved in Smith \cite{S95}, Smith and Waltman \cite{SW95} and Hirsch and Smith \cite{HS05}; while an analog for discrete-mappings was given in Hirsch and Smith \cite{HS05} and Wang and Jiang \cite{WJ01}. For the sake of clarity, we here just highlight the critical points of this approach for the cooperative case; and the competitive case is analogous. Suppose that the Non-oscillation Principle does not hold. Then this approach enables one to finally obtain  without loss of generality a decreasing-interval $I_1=[a,c]$ and a steeply increasing-interval $I_2=[c,d]$, which are attached to each other at $c$, such that $a$ is the only point $s\in I_1$ with $x(s)\geqslant x(d)$.
Based on such a key fact, one can translate the interval $I_1$ to the right by a distance $\delta>0$ as $[a+\delta,c+\delta]$, which is still a decreasing interval, such that either (i) $a+\delta=c$, if  $\abs{I_1}\leqslant \abs{I_2}$; or (ii) $c+\delta=d$, if $\abs{I_1}>\abs{I_2}$. Here $\abs{I_i}$ is the interval length of $I_i$, $i=1,2,$ respectively.
Note that (i) contradicts the property of $I_2$; and (ii) contradicts the property of $I_1$.
Then the Non-oscillation Principle is obtained.

%\begin{figure*}
% \newcommand{\imgwidth}{4in}
%  \newcommand{\imgspace}{-10pt}
%    \centerline{
%    \includegraphics[width=\imgwidth]{acd1.pdf}
%  }
%  \caption{The intervals $I_{k}=[t_{0}+kA, b+t_{*}+kA]$, and the time points $t_{k}=t_{0}+kA$, $k=0,1,\cdots$. The broken line reverse ``L" denotes the corresponding interval is increasing.
%  }\label{fig:acd1}
%\end{figure*}

%\begin{figure*}
%  \newcommand{\imgwidth}{3in}
%  \newcommand{\imgspace}{-10pt}
%    \centerline{
%    \includegraphics[width=\imgwidth]{acd2a.pdf}
%    \includegraphics[width=\imgwidth]{acd2b.pdf}
%  }\caption{acd2a}\label{fig:acd}
%\end{figure*}

The second approach (i.e., Continuation Argument Approach) is also due to Hirsch \cite[Lemma 6.1]{H88}. Nevertheless, compared to the first approach, it actually has received little attention to date.
%We see  Tak\'{a}\v{c} \cite[Lemma 4.8]{T92} for an analog for maps).
Here, again we just mention its key points for the cooperative case. Suppose that the Non-oscillation Principle does not hold. Then one can assume without loss of generality that there exist a decreasing-interval $J_1=[a,b]$ and a steeply increasing-interval $J_2=[c,d]$ with $b<c$.
A critical insight discovered by Hirsch (we here call as Hirsch's Continuation Argument) is that, for any nonnegative integer $n\in \mathbb{N}$, the right-extended interval $[c,d+n\abs{J_2}]$ of $J_2$ is still steeply increasing. Based on this argument, one can directly translate the decreasing interval $J_1$ to the right so that it is plugged into $[c,d+N\abs{J_2}]$ and shares the common left-endpoint of $[c,d+N\abs{J_2}]$ for some $N\in \mathbb{N}$. This then contradicts the steeply-increasing property of $[c,d+N\abs{J_2}]$, which implies the Non-oscillation Principle.

For the eventually cooperative systems, however, the situation has been changed when considering the Non-oscillation Principle. As a matter of fact,  one will encounter difficulties from both of the approaches, respectively. In the first approach,  it deserves to point out that the interval length $\abs{I_1},\abs{I_2}$ mentioned there could be strictly less than $t_*$. Hence, in either (i) or (ii),  the translation distance $\delta=\min\{\abs{I_1},\abs{I_2}\}<t_*$, by which one can not expect to preserve the decreasing property of $I_1$ at all. When such situation occurs, the Finite-Interval Approach becomes invalid completely. As for the second approach, it turns out that the insight from Hirsch can not remain correct anymore. In fact, due to the existence of $t_*$, the right-extended interval $[c,d+n\abs{J_2}]$ could not be steeply increasing even for any $n\geqslant 1$. Consequently, when
dealing with the eventually cooperative or competitive systems, one will run into the obstacles in the second approach as well.

In the present paper, we will first prove the Non-oscillation Principle for eventually cooperative or competitive systems (see Theorem \ref{nonoscillation thm}) in Section 2. As we mentioned above, the popular Finite-Interval Approach becomes invalid completely.
Although one will encounter difficulties from the second approach as well, Hirsch's continuation argument shed a light on showing the non-oscillating property of the orbits for eventually cooperative or competitive systems. We will modify Hirsch's continuation argument to prove the Non-oscillation Principle. Roughly speaking, without of loss of generality we will construct a sequence $\{I_n\}_{n=1}^\infty$ of pairwise disjoint intervals with identical length from the left to the right, such that the minimal left interval $I_1=[a_1,b_1]$ is steeply-increasing and any other $I_n$ contains no point $s$ satisfying $x(s)\leqslant x(a_1)$. Based on such construction, we will further show that any decreasing interval may finally intersect some $I_n$ after certain translations. Such new argument will essentially enable us to succeed in proving the Non-oscillation Principle for eventually cooperative or competitive systems.

Finally, in Section 3, we will utilize the Non-oscillation Principle obtained in Section 2 to prove the non-ordering of (both $\omega$- and $\alpha$-) limit sets of eventually competitive and cooperative systems in $n$-spaces, by which we will further establish the Poincar\'{e}-Bendixson Theorem and structural stability for three-dimensional eventually competitive and cooperative systems.

%\textit{\textbf{Keywords:}} xxxx
%------------------------SECTION DIVIDE LINE-------------------
%\section{Definitions and Notations}
\section{Non-oscillation Principle and its Proof}

We first introduce some basic definitions and notations.
A nonempty closed set $C\subset \mathbb{R}^{N}$ is called a convex cone if it satisfies $C + C \subset C$, $\alpha C\subset C$ for all $\alpha \geqslant 0$, and $C \cap (-C) = \{0\}$. A convex cone is {\it solid} if $\text{Int} C \neq \emptyset$.
We write
\begin{alignat*}{2}
x &\leqslant y \quad& \text{if}&\ \ y - x \in C, \\
x &< y & \text{if}&\ \ y-x  \in C\backslash \{ 0\}, \\
x &\ll y & \text{if}&\ \ y-x \in \text{Int} C.
\end{alignat*}
Notations such as $x\geqslant (>,\gg)\, y$ have the natural meanings. A subset $U\subset \mathbb{R}^{N}$ is \textit{$p$-convex} if $x \leqslant y$ and $x, y\in U$ imply that $U$ contains the line segment between $x$ and $y$.

Let $\phi: \mathbb{R}\times X \to X$ be a flow on an open $p$-convex subset $X\subset\RR^N$.
A \textit{positive orbit} of $x$ is denoted by $\textit{O}^{+}(x)=\{\phi_{t}(x):t\geqslant 0\}$. A \textit{negative orbit} of $x$ is denoted by $\textit{O}^{-}(x)=\{\phi_{t}(x):t\leqslant 0\}$. The \textit{complete orbit} of $x$ is defined as $\textit{O}(x)=\textit{O}^{+}(x)\cup \textit{O}^{-}(x)$. An \textit{equilibrium point} is a point $x$ for which $\textit{O}(x)=\{x\}$. We denote by $E$ the set of all equilibrium points.

The flow $\phi_{t}$ is \textit{eventually cooperative} (resp. {\it eventually competitive}), if there exists some $t_{*}\geqslant 0$ such that $\phi_{t}(x)\leqslant \phi_{t}(y)$ whenever $x\leqslant y$ and $t\geqslant t_*$ (resp. $t\leqslant -t_*$). Note that $\phi_{t}$ is a flow. Then it is also easy to see that any eventually cooperative (resp. eventually competitive) flow  $\phi_{t}$ automatically satisfies the property that $\phi_{t}(x)<_r \phi_{t}(y)$ whenever $x<_r y$ and $t\geqslant t_*$ (resp. $t\leqslant -t_*$), where $<_r$ denotes any one of the relations ``$<,\ll$".

In particular,  $\phi_{t}$ is called \textit{cooperative} (resp. {\it competitive}) if $\phi_{t}$ is eventually cooperative (resp. eventually competitive) with $t_*=0$.

In the following context, we also write $x(t)$ as the trajectory of $x$, for brevity.
A time-interval $J=[a,b]\subset \mathbb{R}$ is called an {\it increasing} interval if $x(a)<x(b)$ and a {\it decreasing} interval if $x(a)>x(b)$. $J$ is {\it steeply increasing} for the trajectory of $x$ if $J$ is increasing and the ending point $a$ is the only point $t\in J$ with $x(t)\leqslant x(a)$.

\begin{thm}\label{nonoscillation thm}
{\rm (Non-oscillation principle)} Assume that $\phi_{t}$ is eventually cooperative or eventually competitive. Then any trajectory $x(t)$ cannot have both an increasing interval and a decreasing interval, provided that $x(t)$ has the complete orbit.
\end{thm}
\begin{proof}
We here focus on the eventually cooperative system; while the eventually competitive system is analogous.

Suppose that there exist both a decreasing interval $J=[c,d]$ and an increasing interval $I=[a,b]$. Without loss of generality, we also assume that $d\leqslant a$. We write the interval length as $|J|=B$, $|I|= A$.

Firstly, we consider the interval $[a+t_{*}, b+t_{*}]$. Clearly, $[a+t_{*}, b+t_{*}]$ is an increasing interval.
Let $t_{0}=\sup \{ t\in [a+t_{*}, b+t_{*}] : x(t)\leqslant x(a+t_{*})\}$, and write $I_{0}=[t_{0}, b+t_{*}]$. So, $I_{0}$ is a steeply increasing interval and its length $|I_{0}|>0$.
Let $I_{1}=[t_{0}+A, b+t_{*}+A]$. We assert that
\begin{equation}\label{E:Insteeply-1}
\text{$I_{1}$ contains no point $s$ satisfying} \ \ x(s) \leqslant x(t_{0}).
\end{equation}
Indeed, suppose that there exists some $s\in I_{1}$ such that $x(s)\leqslant x(t_{0})$. Then $s-A\in I_{0}$; and hence, $s-A-a\geqslant t_{0}-a\geqslant t_{*}$. So, the translation from $[a,b]$ to $[s-A,s]$ yields that $x(s-A)<x(s)$, which implies that $x(s-A)< x(t_{0})$, a contradiction to the definition of $t_{0}$. Thus, we have obtained \eqref{E:Insteeply-1}.

Now, we define $I_{n}=[t_{0}+nA, b+t_{*}+nA]$, for $n=1,2,\cdots.$ Clearly, $|I_{n}|=|I_{0}|>0$. We show that
\begin{equation}\label{E:Insteeply}
\text{$I_{n}$ contains no point $s$ satisfying}\ \ x(s) \leqslant x(t_{0}),\quad \text{for any } n\geqslant 1.
\end{equation}
We will prove \eqref{E:Insteeply} by induction. By \eqref{E:Insteeply-1}, it is clear that \eqref{E:Insteeply} hold for $n=1$. Let \eqref{E:Insteeply} hold for $k=n-1$. To prove \eqref{E:Insteeply} for $k=n$, let us suppose that there exists $s\in I_{n}$ such that $x(s)\leqslant x(t_{0})$. Then $s-A\in I_{n-1}$ and $s-A-a>t_{0}-a\geqslant t_{*}$. Consequently, we translate $[a,b]$ to $[s-A,s]$ to obtain that $x(s-A)<x(s)$. So, $x(s-A)< x(t_{0})$, which  contradicts \eqref{E:Insteeply} for $k=n-1$. Thus, we have proved \eqref{E:Insteeply}.

For brevity, we hereafter write $t_n=t_0+nA$ and let $E=\abs{I_0}=\abs{I_n}$ for $n\geqslant 1$. So, in the context, one can also rewrite $I_n=[t_n,t_n+E]$ for $n\in \mathbb{N}$ (see {\bf Figure \ref{fig:ab}}). Here, $\mathbb{N}$ denotes the set of all nonnegative integers.
\begin{figure*}
  \newcommand{\imgwidth}{4in}
  \newcommand{\imgspace}{-10pt}
    \centerline{
    \includegraphics[width=\imgwidth]{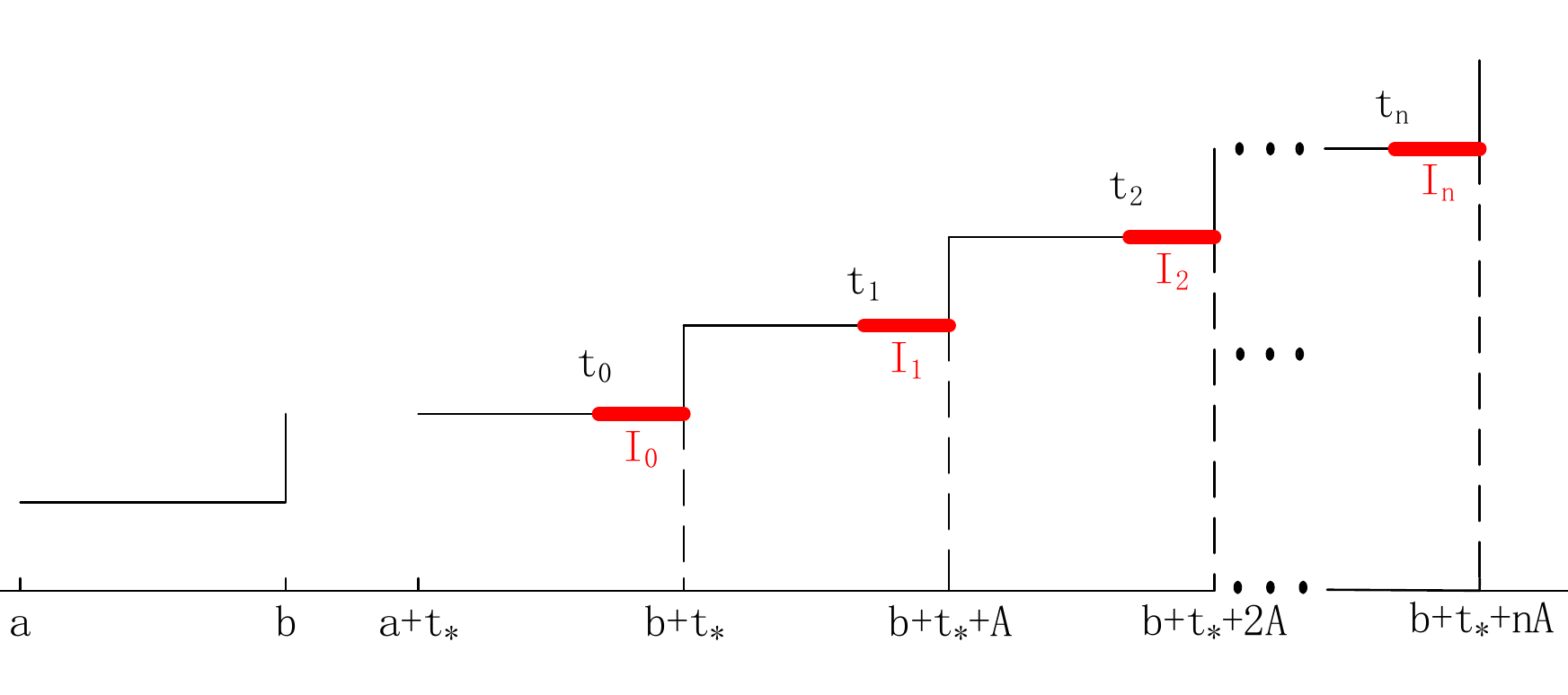}
  }
  \caption{The intervals $I_{n}=[t_{0}+nA, b+t_{*}+nA]=[t_n,t_n+E]$ for $n\in \mathbb{N}$, where $E=\abs{I_n}=\abs{I_0}$ and
   $[a,b]$ is the increasing interval.
  }\label{fig:ab}
\end{figure*}

Secondly, let $c_n=t_0+nB$, for $n\geqslant 1$. One can translate the decreasing interval $[c,d]$ to the right (because the translation distances $\min_{n\geqslant 1}\{c_n-c\}> t_0-c > t_0-a\geqslant t_*$) to obtain that
\begin{equation}\label{E:deceasing}
x(t_0)>x(c_1)>x(c_2)>\cdots>x(c_n)>x(c_{n+1})>\cdots.
\end{equation}
In the following, we {\it claim that one can finally find an index $l_*\geqslant 1$ such that $c_{l_*}\in I_{n_*}$ for some $n_*\in \mathbb{N}$}. Note that $x(t_0)>x(c_{l_*})$ from \eqref{E:deceasing}. Note that this claim either contradicts to \eqref{E:Insteeply}$_{n_*}$, or contradicts to the steeply-increasing property of $I_0$. Thus, one can complete the proof.

So, it remains to prove the claim.
For this purpose, we write $B=k_0A+R_0$, with some $k_0\in \mathbb{N}$ and $0\leqslant R_0<A$.
If $0\leqslant R_0\leqslant E$, then $c_{1}\in [t_{k_0},t_{k_0}+E]=I_{k_0}$. By choosing $l_*=1$ and $n_*=k_0$, we've done.

If $E<R_0<A$, then we write $D_0=A-R_0$. Clearly, $0<D_0<A-E$ and
\begin{equation}\label{E:c-n-expr}
c_n=t_0+nB=t_0+n(k_0+1)A-nD_0=t_{n(k_0+1)}-nD_0
\end{equation}
for any $n\geqslant 1$. Choose $n_0\geqslant 1$ such that
\begin{equation}\label{E:n0-property}
(n_0-1)D_0<A-E\leqslant n_0D_0.
\end{equation}

Case (I). When $n_0D_0\in [A-E,A]$, it follows from \eqref{E:c-n-expr} that $$c_{n_0}\in \left[t_{n_0(k_0+1)-1},t_{n_0(k_0+1)-1}+E\right]=I_{n_0(k_0+1)-1}.$$ Thus, by choosing $l_*=n_0$ and $n_*=n_0(k_0+1)-1$, we've done again (For instance, see {\bf Figure \ref{fig:abcd1}} with $n_{0}=3,k_{0}=1,l_{*}=3$ and $n_{*}=5$).
\begin{figure*}
  \newcommand{\imgwidth}{4.5in}
  \newcommand{\imgspace}{-10pt}
    \centerline{
    \includegraphics[width=\imgwidth]{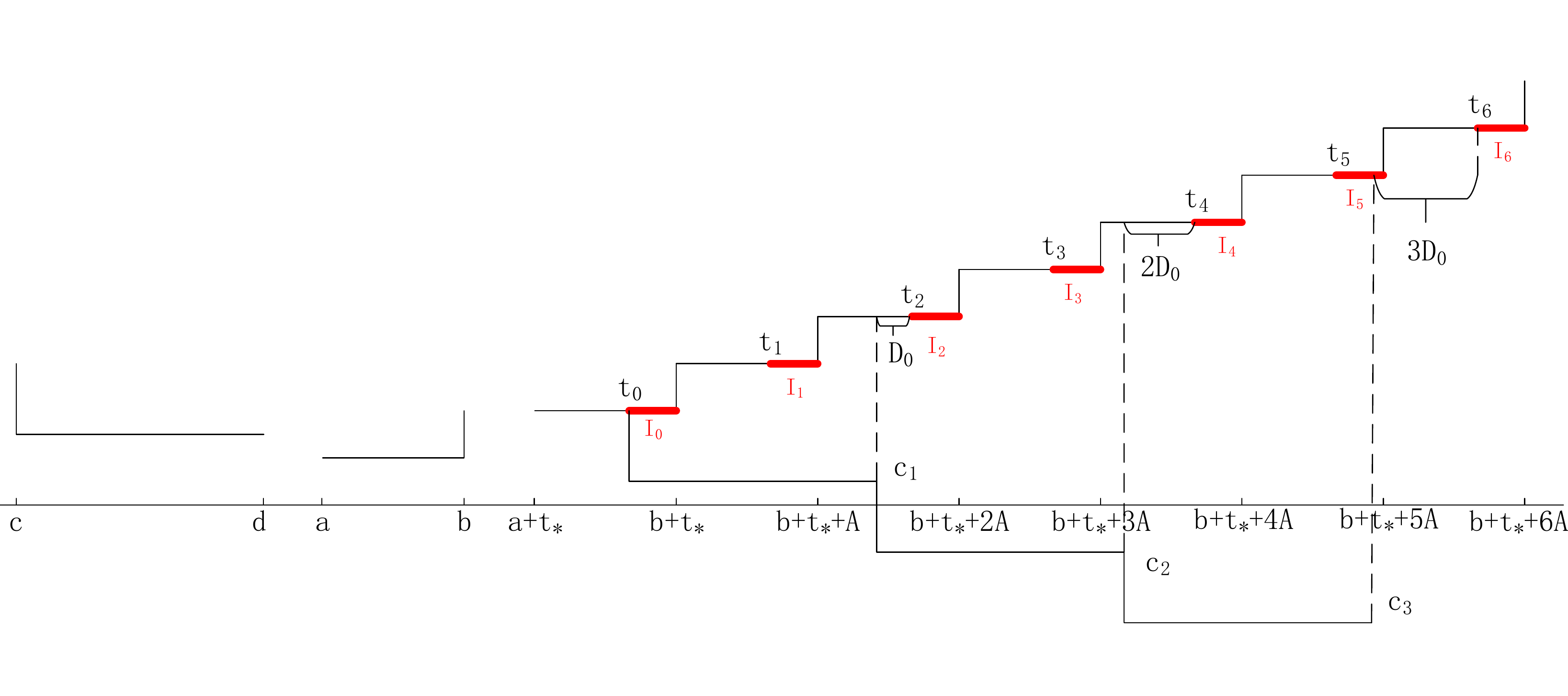}
  }
  \caption
  {
  Case (I) with $n_{0}D_{0}\in [A-E,A]$, where $c_{n_{0}}\in [t_{n_{0}(k_{0}+1)-1},t_{n_{0}(k_{0}+1)-1}+E]=I_{n_0(k_0+1)-1}$ with $n_{0}=3,k_{0}=1,l_{*}=3$ and $n_{*}=5$.
  }\label{fig:abcd1}
\end{figure*}

\vskip 1mm
Case (II). When $n_0D_0>A$, we let $D_1=n_0D_0-A$. A direct calculation from \eqref{E:c-n-expr} and the first inequality in \eqref{E:n0-property} yields that
$$0<D_1<D_0-E \,\text{  and  }\,c_{n_0}=t_{n_0(k_0+1)-1}-D_1;$$
\begin{figure*}
  \newcommand{\imgwidth}{4.5in}
  \newcommand{\imgspace}{-10pt}
    \centerline{
    \includegraphics[width=\imgwidth]{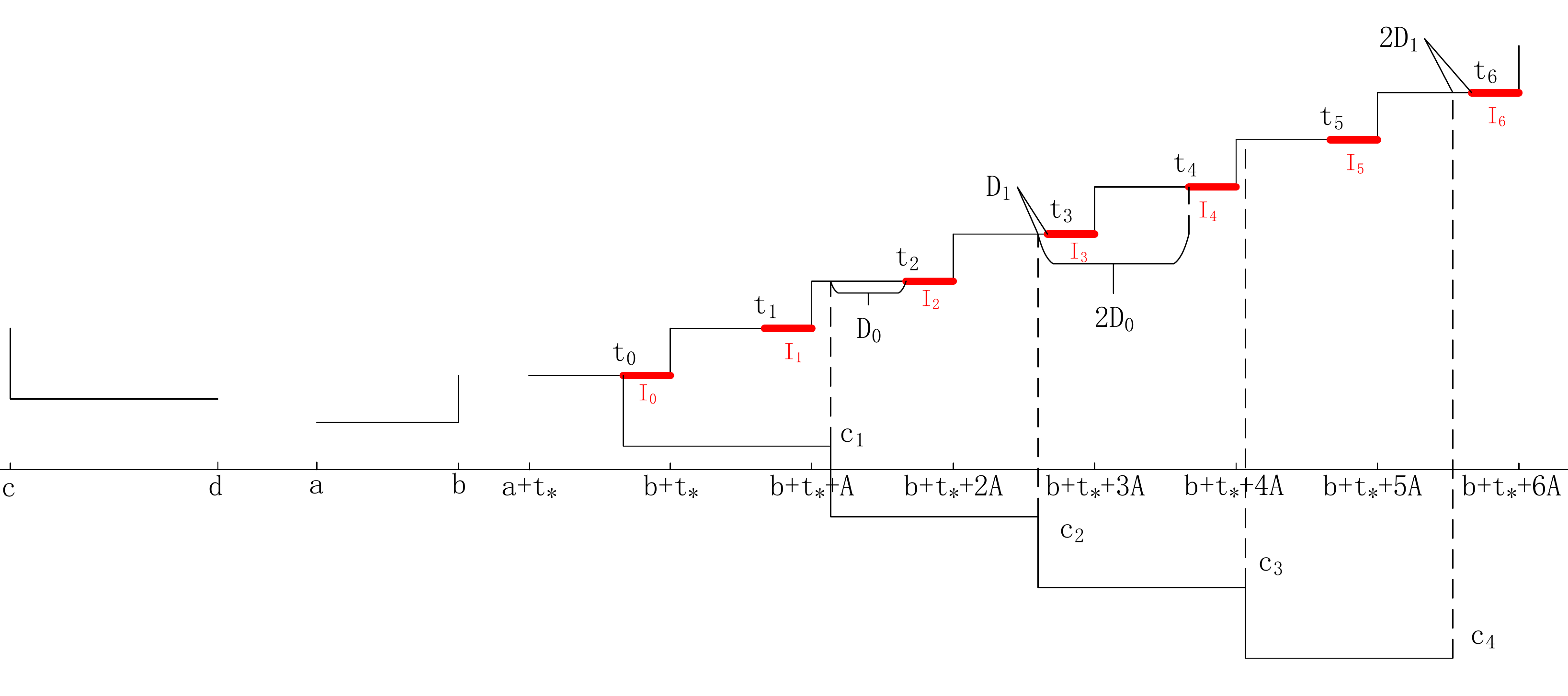}
  }
  \caption{
  Case (II) with $n_{0}D_{0}>A$, where $c_{n_{0}}=t_{n_0(k_0+1)-1}-D_1\in (t_{n_{0}(k_{0}+1)-2}+E, t_{n_{0}(k_{0}+1)-1})$ and $c_{n\cdot n_0}=t_{n\cdot[n_0(k_0+1)-1]}-nD_1$, with $n_{0}=n=2$, $k_{0}=1$ and $D_1=n_0D_0-A$.
  }\label{fig:abcd2}
\end{figure*}
and moreover, we have
\begin{equation}\label{E:nn0-prop}
c_{n\cdot n_0}=t_{n\cdot[n_0(k_0+1)-1]}-nD_1, \,\,\text{for any }n\geqslant 1.
\end{equation}
(For instance, see {\bf Figure \ref{fig:abcd2}}, for $n_{0}=n=2$, $k_{0}=1$ and $D_1=n_0D_0-A$).
For such $D_1$, similarly as in \eqref{E:n0-property}, one can
choose $n_1\geqslant 1$ such that
\begin{equation}\label{E:n1-property}
(n_1-1)D_1<A-E\leqslant n_1D_1;
\end{equation}
and moreover, by \eqref{E:nn0-prop}, we can also follow the same argument after \eqref{E:n0-property} to obtain either

(i). $c_{n_1\cdot n_0}\in I_{n_1\cdot[n_0(k_0+1)-1]-1}$, whenever $n_1D_1\in [A-E,A]$ (Thus, we've done by choosing $l_*=n_1\cdot n_0$ and $n_*=n_1\cdot [n_0(k_0+1)-1]-1$); or otherwise,

(ii).  $c_{n_1\cdot n_0}=t_{n_1\cdot[n_0(k_0+1)-1]-1}-D_2,$ where $D_2=n_1D_1-A>0$ (Hence, $0<D_2<D_1-E<D_0-2E$).

By repeating the similar argument, if necessary, we can choose two sequences $\{n_{i}\}$ and $\{D_{i}\}$, $i\in \mathbb{N}$, satisfying $$D_{i+1}=n_{i}D_{i}-A,\,\,\,(n_{i}-1)D_{i}<A-E\,\, \text{ and }\,\,0<D_{i+1}<D_{i}-E.$$
Let $l_{i}=\prod_{j=0}^in_j$. Then one has $c_{l_{i}}=t_{h_{i}}-D_{i+1}$, where $h_i$ satisfies
the recursion $h_{i+1}=n_{i+1}\cdot h_{i}-1$ with $h_{0}=n_{0}(k_{0}+1)-1$. Now, we write $D_{0}=pE+F$ with some integer $p\geqslant 1$ and some $F\in (0,E]$. Then, one has $0<D_{p}<D_{0}-pE=F\leqslant E$. As a consequence, there exists an integer $n_{p}\geqslant 1$ such that $n_{p}D_{p}\in [A-E,A]$. So, by choosing $l_{*}=n_{p}\cdot l_{p-1}$ and $n_{*}=h_p=n_{p}\cdot h_{p-1}-1$, we have $c_{l_{*}}\in I_{n_{*}}$. Thus, we've proved the claim.
\end{proof}

\begin{rmk}
Motivated by the proof of Theorem  \ref{nonoscillation thm}, one can further obtain the {\it Non-oscillation Principle} for the {\it discrete-time} eventually competitive or cooperative systems.  As a matter of fact, the proof is even more straightforward for the discrete-time systems.
Here we just mention the discrete-time eventually competitive systems, for instance. The discrete-time eventually cooperative systems are analogous.

Let $X\subset\RR^N$ be an open subset and $T:X\to T(X)\subset\mathbb{R}^N$ be a homeomorphism (For example, $T$ is the Poincar\'{e} map associated with a time-periodic ODE system on $X$). We denote
 $\{T^{m}z: m\in \mathbb{Z}\}$ the orbit of $z\in X$. $T$ is said to be eventually competitive if there exists an integer $n_*\geqslant 1$ such that
\begin{equation}\label{T-even-comp}
x<y \,\text{ whenever } \,T^{n}x<T^{n}y\, \text{ with some }\, n \geqslant n_*.
\end{equation}
Let $[m,n]=\{m,m+1,\cdots,n\}\subset \mathbb{Z}$ be an integer segment. $[m,n]$ is said to be an increasing (resp. decreasing) segment if $T^{m}z< ({\rm resp. >})\, T^{n}z$.

In the following, we prove that the orbit of $z$ cannot have both an increasing segment and a decreasing segment, provided that $z$ has the complete orbit. Without loss of generality, suppose that the orbit of $z$ contains a decreasing segment $[m,n]$ and an increasing segment $[k,l]$ with $m<n<k<l$. Let $A=n-m, B=l-k$, then $A,B> 0$ and there exist integers $p,q$ such that $pA=qB$. By translating the $[m,n]$ to the left, it follows from eventually competitive property that there are decreasing segments $I_{i}=[n-n_*-(i+1)A,n-n_*-iA]$, for $i=0,1,\cdots,p-1$, where $n_*$ is as in \eqref{T-even-comp}. In particular, one has $T^{n-n_*-pA}z>T^{n-n_*}z$.

On the other hand, since $l-(n-n_*)>n_*$, we translate increasing segment $[k,l]$ to the segments $[c_j,c_{j-1}]$, $j=1,2,\cdots,q$, so that
$T^{c_{j}}z<T^{c_{j-1}}z$. Here $c_{j}=n-n_*-jB$, for $j=0,1,2,\cdots,q$. As a consequence, one has $T^{n-n_*-qB}z<T^{n-n_*}z$.
Noticing $pA=qB$, we have obtain a contradiction. Thus, we have proved Non-oscillation Principle for the discrete-time eventually competitive systems.
\end{rmk}

\section{Eventually competitive and cooperative systems: Limit Sets, Poincar\'{e}-Bendixson Theorem and Structural Stability}
In this section, we will utilize the {\it Non-oscillation Principle} established in Section 2 to prove non-ordering of (both $\omega$- and $\alpha$-) limit sets of eventually competitive and cooperative systems in $N$-spaces, as well as the Poincar\'{e}-Bendixson Theorem and structural stability for 3-dimensional eventually competitive and cooperative systems.

As mentioned in the introduction, Non-ordering of Limit Sets was regarded as the fundamental building block of the theory of competitive and cooperative systems (see, e.g. \cite{S17}). For eventually cooperative systems, Hirsch \cite{H85} has proved the non-ordering of the $\omega$-limit sets by using
the monotone convergence criterion (see \cite[Theorem 2.2]{H85}). However, it {\it remains unknown} for the non-ordering of the $\omega$-limit sets for eventually competitive systems; or equivalently, the $\alpha$-limit sets for eventually cooperative systems. Under such situations, the monotone convergence criterion approach does not work anymore.

By virtue of Non-oscillation Principle (Theorem \ref{nonoscillation thm}) established in Section 2, we present the following theorem on Non-ordering of (both $\omega$- and $\alpha$-) limit sets for eventually competitive and cooperative systems. Hereafter, we always mean the limit sets by both $\omega$- and $\alpha$- limit sets.

\begin{thm}\label{nonordering thm}
{\rm (Non-ordering of limit sets)} Any limit set $L$ of an eventually competitive or cooperative system cannot contain two points related by $\ll$.
\end{thm}
\begin{proof}
We here just consider the $\omega$-limit set of an eventually competitive system; and other cases are analogous. Suppose that $L$ contains $y$ and $z$ such that $y\ll z$, then there exist some neighborhoods $U$ and $V$ of $y$ and $z$, respectively; such that $\tilde{y}\ll \tilde{z}$ for any $\tilde{y}\in U$ and $\tilde{z}\in V$. Then one can choose $t_{1}, t_{2}, t_{3}, t_{4}>0$ with $t_{1}<t_{2}<t_{3}<t_{4}$ such that $x(t_{1}), x(t_{4})\in U$ and $x(t_{2}), x(t_{3})\in V$. So, the interval $[t_{1},t_{2}]$ is increasing and the interval $[t_{3},t_{4}]$ is decreasing, which contradicts Theorem \ref{nonoscillation thm}.
\end{proof}

\begin{rmk}\label{nonordering of periodic orbit}
In particular, for any nontrivial periodic orbit $\g$ of an eventually competitive or cooperative system, Non-oscillation Principle (Theorem \ref{nonoscillation thm}) guarantees that $\g$ cannot contain two points that are related by ``$<$". Since the proof is same as that of \cite[Prop.3.3.3]{S95}, we omit it here.
\end{rmk}

An instant consequence of Non-ordering of limit sets (Theorem \ref{nonordering thm}) is the following:

\begin{thm}\label{topol-equivalent}
The flow on a compact limit set $L$ of an eventually competitive or cooperative system is topologically equivalent to a flow on a compact invariant set of a Lipschitz systems of differential equation in $\mathbb{R}^{N-1}$.
\end{thm}
\begin{proof}
The proof is similar to the proof of Theorem \cite[Theorem 3.3.4]{S95}. In fact, one can take the mapping
\begin{equation}\label{theta}
\Theta: \mathbb{R}^N\to H_{v}; x\mapsto x-(x,v)v;
\end{equation}
where $v$ lies in the interior of the cone and $H_{v}$ denotes the hyperplane orthogonal to $v$.
By virtue of Theorem \ref{nonordering thm}, one has $\Theta(a)\neq \Theta(b)$ for any distinct $a,b\in L$, which means the restriction of $\Theta$ on $L$ is a one-to-one map. Then one can repeat the argument in the proof of Theorem \cite[Theorem 3.3.4]{S95} to obtain that $L$ is topologically conjugate to a compact invariant set of a Lipschitz-continuous vector field in $\mathbb{R}^{N-1}$.
\end{proof}
\vskip 3mm

In the following, we will focus on 3-dimensional eventually competitive and cooperative systems.
We will present the corresponding Poincar\'{e}-Bendixson Theorem and structural stability for 3-dimensional eventually competitive and cooperative systems.

Before doing that, we give some notations. Let $L,M\subset \mathbb{R}^{N}$ and we write $L\preceq M$, if for any $z\in L$ there exists $w\in M$ such that $z\leqslant w$, and vice versa. In particular, when $M=\{p\}$ is a singleton, $L\preceq \{p\}$ means that $z\leqslant p$ for any $z\in L$; by which we also write as $L \leqslant p$.
We further write $L\ll p$, if $z\ll p$ for any $z\in L$. Notations such as $L\succeq M,L \geqslant p$  have the natural meanings.

\begin{thm}\label{P-B thm}
{\rm (Poincar\'{e}-Bendixson Theorem)} A compact limit set $L$ of an eventually competitive or cooperative system in $\mathbb{R}^{3}$ that contains no equilibrium points is a periodic orbit.
\end{thm}

\begin{proof}
The proof is based on our Theorem \ref{topol-equivalent} and the argument in \cite[Section 4]{S95} (see also \cite[Thoerem 1]{H90}) for competitive and cooperative systems. For the sake of completeness, we
give more detail.

Again, we focus on the eventually competitive case; while the eventually cooperative case is analogous. For explicitness, we let $L=\omega(x)$. By Theorem \ref{topol-equivalent}, the flow $\phi_t$ on $L$ is  topologically conjugate to a flow $\psi_t$ on the compact invariant set $\Theta(L)\subset\mathbb{R}^2$, where $\Theta$ is defined as in \eqref{theta}. As a consequence, $\Theta(L)$ is a chain-recurrent set with respect to $\psi_t$, because $L$ is a chain-recurrent set for $\phi_t$. Recall that $\Theta(L)$ contains no equilibrium of $\psi_t$ (since $L$ contains no equilibrium). Then, by following the same argument in \cite[Theorem 4.1]{S95}, the chain-recurrence of $\Theta(L)$ implies that it is either a single periodic orbit or an annulus of periodic orbits.

It remains to rule out that $\Theta(L)$ is an annulus of periodic orbits. Suppose that $\g$ is a periodic orbit in $L$ with $\Theta(\g)\subset {\rm Int}(\Theta(L))\subset \mathbb{R}^2$. Then the periodic orbit $\Theta(\g)$ separates $\Theta(L)$ into two components. Fix $a,b\in L$ such that $\Theta(a),\Theta(b)$ belong to the different component of $\Theta(L)\setminus \Theta(\g)$. Since $\{\phi_t(x)\}_{t\geqslant 0}$ repeatedly revisit the neighborhoods of $a$ and $b$, $\Theta(\phi_t(x))$ will intersect $\Theta(\g)$ at a sequence $t_k\to \infty$. So, for each $k\geqslant 1$, let $z_k\in \g$ be such that $\Theta(\phi_{t_k}(x))=\Theta(z_k)$; and hence, one may choose a subsequence, if necessary, such that $\phi_{t_k}(x)\ll z_k$ (resp. $z_k\ll \phi_{t_k}(x)$) for all $k\geqslant 1.$  Therefore, for any $s>0$, if we choose $t_k$ so large that with $s-t_k\leqslant -t_*$, then the {\it eventually competitive property} implies that $\phi_{s}(x)=\phi_{s-t_k}(\phi_{t_{k}}x)< (\text{resp. } >)\, \phi_{s-t_k}(z_{k})\in \g$. Given any $y\in L$, let $\tau_{n}\to \infty$ be such that $\phi_{\tau_{n}}(x)\to y$. Then, for such $\tau_{n}$, there exists $w_{n}\in \g$ such that $\phi_{\tau_{n}}(x)< (\text{resp. } >)\, w_{n}$. By choosing a subsequence if necessary, we obtain $w_{n}\to w\in \g$ and $y\leqslant (\text{resp. } \geqslant)\, w$. Hence, by arbitrariness of $y$, one has $L\preceq \g$ or $L\succeq \g$.

Since $\Theta(L)$ is an annulus of periodic orbits, one can choose three different periodic orbits $\g_{i}\in L$ such that $\Theta(\g_{i})\subset {\rm Int}(\Theta(L))$, $i=1,2,3$. Moreover, each periodic orbit $\g_{i}$ has the property that either $L\preceq \g_{i}$ or $L\succeq \g_{i}$. Consequently, one can consider without loss of generality the first two periodic orbits $\g_{1}, \g_{2}$ satisfying $L\preceq \g_{1}$ and $L\preceq \g_{2}$.
So, let some $u\in \g_{1}\subset L$, then there exists $w\in \g_{2}$ such that $u < w$. Note also $w\in \g_{2} \subset L$, then there exists $v\in \g_{1}$ such that $w < v$. Thus, one has $u, v \in \g_{1}$ satisfying $u < v$, which contradicts Remark \ref{nonordering of periodic orbit}.
\end{proof}

Motivated by the work of Hirsch \cite{H90}, we will discuss at the end of this section the structural stability for 3-dimensional eventually competitive and cooperative systems.

Let $M\subset X (\subset \mathbb{R}^3)$ be a smooth compact manifold with boundary $\partial M$, and $\mathscr{F}^{1}(M)$ denote the space of $C^{1}$ vector fields on $M$ which are transverse to $\partial M$. A vector field $H\in \mathscr{F}^{1}(M)$ is called \textit{structurally stable}, if there exists a neighborhood $U$ of $H$ in $\mathscr{F}^{1}(M)$ such that for any $G\in U$ there is a homeomorphism $g$ taking $H$-orbits to $G$-orbits.

For brevity, we hereafter use the term ``cycle" refers to a nontrivial periodic orbit of $H$.\vskip 1mm

A vector field $H\in \mathscr{F}^{1}(M)$ satisfies the \textit{Morse-Smale} conditions, if the following conditions hold (see, e.g. Smale \cite{Smale60} or Hirsch \cite[p.1231]{H90}):

(H1) All equilibria and cycles are hyperbolic and their stable and unstable manifolds intersect only transversely;

(H2) The numbers of equilibria and cycles are finite;

(H3) Every limit set is an equilibrium or a cycle.

\vskip 1mm
\noindent In particular, $H$ is called \textit{Kupka-Smale}, if only (H1) holds.
\vskip 2mm

Morse-Smale conditions are known to be the sufficient conditions for structural stability (see, e.g. \cite{Palis1969}); and structurally stable vector fields are Kupka-Smale (see, e.g. \cite{Markus1961}). Although Kupka-Smale conditions are generic, the structurally stable vector fields are not generic when the dimension of the system is larger than 2 (see, e.g. \cite{Newhouse1970,Newhouse1980}). On the other hand, Hirsch \cite{H90} succeeded in proving that, for cooperative or competitive systems on $\mathbb{R}^{3}$, $H$ is Kupka-Smale if and only if it satisfies Morse-Smale conditions; Consequently, for cooperative or competitive $3$-dim systems, Kupka-Smale (and hence, Morse-Smale) concepts coincide with the structural stability.

Based on all our previous results, we will discuss that, for {\it eventually cooperative or competitive systems} generated by $C^1$ vector field $F$ on $X\subset \RR^3$, $F|_M$ is Kupka-Smale if and only if it satisfies Morse-Smale conditions on $M$; and hence, for such systems, Kupka-Smale (and Morse-Smale) concepts coincide with the structural stability (see Theorem \ref{structural stability thm} below).

As we mentioned in the introduction, eventually cooperative (or competitive) systems are
NOT limited to near cooperative (or competitive) systems in the context of perturbation theory. Therefore, the results presented here exhibit the significance for developing the theory of eventual cooperative (or competitive) systems.
\vskip 2mm

We first present the following technical lemma, due to Hirsch \cite[p.1230]{H90}, which also remains true for 3-dimensional eventually competitive or cooperative systems. In the following, we call a \textit{circuit} as a sequence of equilibria $p_{0},\dots, p_{n}=p_{0}$, $n\geqslant 1$, such that $W^{u}(p_{i-1})\cap W^{s}(p_{i})\neq \emptyset$, where $W^{u}(p)$ and $W^{s}(p)$ is the unstable and stable manifolds at $p$, respectively.

\begin{lem}\label{porbits-collection}
Let $F$ be a $C^1$-vector field on $X$ which generates an eventually cooperative or competitive system and $F$ is transverse to $\partial M$. Assume that
all equilibria and cycles of $F|_M$ are hyperbolic and there are no circuits. Let $\{C_{n}\}$ be an infinite sequence of distinct cycles in $M$. Then there exists a subsequence of $\{C_{n}\}$, still denoted by $\{C_{n}\}$, satisfying that,
for any $n$ and any $x\in C_{n}$ there exist $m>n$ and $y\in C_{m}$ such that $x\gg y$ or $x\ll y$.
\end{lem}
\begin{proof}
See Hirsch \cite[p.1230, Lemma]{H90}, because the hyperbolicity of cycles implies that for any real number $T>0$, the number of cycles in $M$ having period no more than $T$ is finite (cf. \cite[Theorem 2]{H90}).
\end{proof}

In order to state the following main result in this section, we further assume that the flow $\phi_{t}$ generated by $F$ is \textit{eventually strongly cooperative} (resp. {\it eventually strongly competitive}), that is, $\phi_t$ is eventually cooperative (resp. eventually competitive) and there is a $\tau_{*}>0$ such that $\phi_{t}(x)\ll \phi_{t}(y)$ whenever $x<y$ and $t\geqslant \tau_*$ (resp. $t\leqslant -\tau_*$) (see e.g., \cite[p.428]{H85} or \cite[p.4]{SM17}).

\begin{thm}\label{structural stability thm}
Let $F$ be a $C^1$-vector field on $X$ which generates an eventually strongly cooperative or competitive system and $F$ is transverse to $\partial M$. If $F|_M$ is Kupka-Smale, then $F|_M$ satisfies Morse-Smale conditions and therefore structurally stable. Conversely, if $F|_M$ is structurally stable, then it satisfies Morse-Smale conditions.
\end{thm}

\begin{proof}
We here consider the eventually strongly competitive system $\{\phi_{t}\}$, while the eventually strongly cooperative system is analogous. We will show the Kupka-Smale condition (H1) will imply (H2)-(H3).

For (H2), the hyperbolicity in (H1) clearly implies that the set of equilibria in $M$ is finite; and for any real number $T>0$, the number of cycles in $M$ having period no more than $T$ is finite. Moreover, the transversality of their stable and unstable manifolds rules out the occurrence of circuits.

We now show that the numbers of cycles are finite. Suppose that in $M$ there is an infinite sequence $\{C_{n}\}$ of distinct cycles, with periods $\{T_{n}\}$. Then, by Lemma \ref{porbits-collection}, one can choose a subsequence of $\{C_{n}\}$, still denoted by $\{C_{n}\}$, with points $x_{n}\in C_{n}$, $v_{n+1}\in C_{n+1}$ such that $x_{n} \ll v_{n+1}$ for all $n$, or $x_{n}\gg v_{n+1}$ for all $n$. Without loss of generality, we assume the former holds.

So, we have two order relationships $x_{n}\ll v_{n+1}$ and $x_{n+1}\ll v_{n+2}$, where $x_{n}\in C_{n}$, $x_{n+1},v_{n+1}\in C_{n+1}$, $v_{n+2}\in C_{n+2}$. Choose some $s\in (0, T_{n+1}]$, $n_{1}\in \mathbb{Z}_{+}$ such that $n_{1}T_{n+1}\geqslant t_{*}$ and $\phi_{-s-n_{1}T_{n+1}}(x_{n+1})=v_{n+1}$. Then it follows from the {\it eventually} competitive property that $\phi_{-s-n_{1}T_{n+1}}(x_{n+1})\ll \phi_{-s-n_{1}T_{n+1}}(v_{n+2})\in C_{n+2}$. Let $\bar{x}_{n+1}=v_{n+1}$, $\bar{v}_{n+2}=\phi_{-s-n_{1}T_{n+1}}(v_{n+2})$. Then we obtain $x_{n}\ll \bar{x}_{n+1}\ll \bar{v}_{n+2}$. By repeating this process, we recursively obtain a sequence, still denoted by $\{x_n\}$, such that $x_{n}\in C_{n}$ and $x_{n} \ll x_{n+1}$ for any $n\geqslant 1$.

Without changing notation, we take a subsequence $\{x_n\}$ such that $x_{n} \to q\in M$.
Then, we {\it claim that the $\alpha$-limit set $\alpha (q)=p\in E$}. Otherwise, there exist $p_{1}, p_{2}\in \alpha (q)$ with $p_{1}\neq p_{2}$. Since $p_{1}\in \alpha(q)$, one can choose two subsequences $t_{k}\to \infty$ and $\{x_{k}\}\subset \{x_{n}\}$ such that $x_{k}\to q$ and $\phi_{-t_{k}}(x_{k})\to p_{1}$ as $k\to \infty$. For $p_{2}\in \alpha(q)$, one can further choose two subsequences $\tau_{m}\to \infty$ and $\{x_{m}\}\subset \{x_{k}\}$ such that $x_{m}\to q$ and $\phi_{-\tau_{m}}(x_{m})\to p_{2}$ as $m\to \infty$. For brevity, we write $z_{m}=\phi_{-t_{m}}(x_{m})$ and $w_{m}=\phi_{-\tau_{m}}(x_{m})$. Clearly, $z_{m},w_{m}\in C_{m}$ with $z_{m}\to p_{1}$ and $w_{m}\to p_{2}$. Since $x_{m}\ll x_{m+1}$ and those cycles are distinct, it then follows from \cite[p.434, Theorem 3.8]{H85} that $C_{m}\ll C_{m+1}$. Here $C_{m}\ll C_{m+1}$ means that $a\ll b$, for any $a\in C_{m}$ and $b\in C_{m+1}$. As a consequence, $z_{m}\ll w_{m+1}$ and $w_{m}\ll z_{m+1}$ hold for all $m$, which implies that $p_{1}\leqslant p_{2}$ and $p_{2}\leqslant p_{1}$. Hence, $p_1=p_2$, a contradiction. This proves the claim.

We will show that $C_{n} \ll p$ for any $n\geqslant 1$. Indeed, for any $z\in C_{n}$, there exists a sequence $t_{i} \to \infty$ such that $\phi_{-t_{i}}(x_{n}) \to z$ as $i\to \infty$. Meanwhile, $\phi_{-t_{i}}(q) \to p$ as $i \to \infty$. Since $x_{n} \ll q$, one has $\phi_{-t_{i}}(x_{n}) \ll \phi_{-t_{i}}(q)$ for all $t_{i}\geqslant t_{*}$. Hence, $z\leqslant p$.  By arbitrariness of $z$ and $n$, we have $x_{n} \ll x_{n+1} \leqslant p$ for all $n$. Moreover, noticing that $z=\phi_{-s-n_{0}T_{n}}(x_{n})$, for some $s\in (0, T_{n}]$ and some $n_{0}$ with $n_{0}T_{n}\geqslant t_{*}$, it follows that $z=\phi_{-s-n_{0}T_{n}}(x_{n}) \ll \phi_{-s-n_{0}T_{n}}(p) = p$, which implies that $C_{n}\ll p$ for all $n\geqslant 1$.

In the following, we will deduce a contradiction to the hyperbolicity of $p$, by showing that any neighborhood $U$ of $p$ will contain some $C_{n} \subset U$. In fact, since $x_{n}\to q$ and $\alpha (q)=p$, there exist two sequences $t_{k}\to \infty$ and $\{x_{n_{k}}\}\subset \{x_{n}\}$ such that $\phi_{-t_{k}}(x_{n_{k}})\to p$ as $k\to \infty$. For brevity, we write $z_{n_{k}}=\phi_{-t_{k}}(x_{n_{k}})$. Clearly, $z_{n_{k}}\in C_{n_{k}}\ll p$ and $z_{n_{k}}\to p$. Choose some $k_{0}>0$ such that $\{x:z_{n_{k_{0}}}\ll x\ll p\}\subset U$.
Let $d$ denote the least upper bound
of $C_{n_{k_{0}}}$ in $X$. Then $C_{n_{k_{0}}}\leqslant d\ll p$. Since $z_{n_{k}}\to p$, one can find $k_{1}>k_{0}$ such that $d\ll z_{n_{k_{1}}}\ll p$. Using eventually competitive property, we obtain $\phi_{-t-n_{1}T_{n_{k_{1}}}}(d) \ll \phi_{-t-n_{1}T_{n_{k_{1}}}}(z_{n_{k_{1}}}) \ll p$ for $\forall t\in (0, T_{n_{k_{1}}}]$ and some $n_{1}$ with $n_{1}T_{n_{k_{1}}}\geqslant t_{*}$. Since $C_{n_{k_{0}}}$ is invariant, we have $C_{n_{k_{0}}}\leqslant \phi_{-t-n_{1}T_{n_{k_{1}}}}(d)$. In particular, $z_{n_{k_{0}}}\leqslant \phi_{-t-n_{1}T_{n_{k_{1}}}}(d)$. Therefore, $z_{n_{k_{0}}}\ll \phi_{-t-n_{1}T_{n_{k_{1}}}}(z_{n_{k_{1}}}) \ll p$ for $\forall t\in (0, T_{n_{k_{1}}}]$, which implies that $C_{n_{k_{1}}}\subset U$. Thus, we have obtained (H2).

To prove (H3), we note that, by Theorem \ref{topol-equivalent}, the flow on the limit set $L$ is topologically equivalent to the Lipschitz planar flow $\psi_t$ on a compact, connected, chain-recurrent invariant set $\Theta(L)$. Together with (H2), the generalized Poincar\'{e}-Bendixson theorem (See, e.g. Hirsch \cite[p.1231, Remark]{H90}) implies that
$\Theta(L)$ consists of a finite number of equilibria, periodic orbits and, possibly, entire orbits whose $\omega$- and $\alpha$-limit sets are periodic orbits or equilibria contained in $\Theta(L)$. By the same argument in \cite[Theorem 4.1]{S95}, the chain-recurrence of $\Theta(L)$ rules out the possibility of the entire orbit whose $\omega$- or $\alpha$-limit set is a periodic orbit. As a consequence, $\Theta(L)$ only consists of a finite number of equilibria, periodic orbits and, possibly, entire orbits connecting equilibria in $\Theta(L)$. Moreover, again by its chain-recurrence, $\Theta(L)$ will contain a circuit whenever it admits entire orbits connecting equilibria. Since the  transversality in (H1) rules out the circuit in $\Theta(L)$, we obtain that $\Theta(L)$ only consists of a finite number of equilibria and periodic orbits; and hence, $\Theta(L)$ is just an equilibrium or a periodic orbit because of the connectedness of $\Theta(L)$, and so is $L$.
Thus, we have completed the proof.
\end{proof}

%\begin{defn} xxxx
%\end{defn}

%\begin{defn} xxxx
%\end{defn}

%\addtocounter{thm}{-3}
%\begin{lem} xxxx
% \end{lem}

%\begin{thm} xxxx
%\end{thm}

%x
%\begin{rmk} XXX
%\end{rmk}

%\begin{figure}[htbp]
%\centering
%%%%\subfigure[Fig.1:$S_{0}$ is the substrate, $S_{2}$ is the product, $E,F$ are two kinds of enzyme]{\includegraphics[height=3.5cm,width=9.5cm]{EG2}}
%\includegraphics[height=8cm,width=15cm]{abcd(1)fig}
%\caption{A substrate $S_{0}$ is ultimately converted into a product $S_{2}$ by an enzyme $E$ and conversely, $S_{2}$ is transformed into the original $S_{0}$ by a second enzyme $F$.}
%\end{figure}

\end{document}